\documentclass[letterpaper, 10 pt, conference]{ieeeconf}  % Comment this line out
\IEEEoverridecommandlockouts                              % This command is only
\usepackage{amsmath} % assumes amsmath package installed
\usepackage{amssymb}  % assumes amsmath package installed
\usepackage{amsfonts}
\usepackage{dsfont}
\makeatletter
\newtheorem{lemma}{Lemma}[section]

\newtheorem{example}[lemma]{Example}
\newtheorem{definition}[lemma]{Definition}

\newtheorem{proposition}[lemma]{Proposition}
\newtheorem{remark}[lemma]{Remark}

\def\X{\mathcal{X}}
\def\U{\mathcal{U}}
\def\Y{\mathcal{Y}}

\def\K{\mathcal{K}}

\newcommand{\eps}{\epsilon}

\def\dx{\delta x}
\def\du{\delta u}
\def\dy{\delta y}

\def\dSigma{\delta \Sigma}

\newcommand{\supscr}[2]{#1^{\textsc{#2}}}
\newcommand{\pder}[2]{\frac{\partial #1}{\partial #2}}
\newcommand{\real}{\mathbb{R}}

\DeclareMathOperator{\rank}{rank}

\title{\LARGE \bf
A geometric approach to differential Hamiltonian systems and differential Riccati equations}

\author{Arjan van der Schaft
\thanks{A.J. van der Schaft is with the Johann Bernoulli Institute for Mathematics and Computer
Science, University of Groningen, PO Box 407, 9700 AK, the
Netherlands
        {\tt\small A.J.van.der.Schaft@rug.nl}}
}

\begin{document}

\maketitle
\thispagestyle{empty}
\pagestyle{empty}

\begin{abstract}
Motivated by research on contraction analysis and incremental stability/stabilizability the study of 'differential properties' has attracted increasing attention lately. Previously lifts of functions and vector fields to the tangent bundle of the state space manifold have been employed for a geometric approach to differential passivity and dissipativity. In the same vein, the present paper aims at a geometric underpinning and elucidation of recent work on 'control contraction metrics' and 'generalized differential Riccati equations'.
\end{abstract}

\section{Introduction}
This paper provides a geometric counterpart to recent work on differential versions of Hamiltonian matrices and Riccati equations, motivated by incremental stability analysis and incremental stabilizability (cf., \cite{lohmiller, pavlov, jouffroy, sepulchreforni, sepulchreforni1}). While in most of this work either coordinate expressions are used (see, e.g., \cite{manchester, manchester1}), or an algebraic point of view is adopted (see in particular \cite{kawanooud, kawano, halas}, motivated by, e.g., \cite{leroy}, \cite{lam,lam1}), the current paper provides a geometric, coordinate-free, description based on the geometric theory of liftings of functions and vector fields on manifolds to their tangent and cotangent bundle as detailed in \cite{Yals}; see also \cite{crouch, cortes, vdsdiffpass}. 

Such a geometric approach may provide additional insights, and may yield elegant proofs for statements which otherwise require cumbersome coordinate computations that are only locally valid. Furthermore, a geometric approach can address {\it global} problems. It was explored before in the study of differential passivity and dissipativity in \cite{vdsdiffpass}.

\section{Preliminaries}
\subsection{Basic notions}
Throughout this paper all objects (manifolds, functions, vector fields, one-forms, (co-)distributions, subbundles, ..) will be assumed to be {\it smooth} (infinitely differentiable). 

Consider an $n$-dimensional state space manifold $\X$ with tangent bundle $T \X$ and co-tangent bundle $T^* \X$. Consider furthermore a vector field $f$ on $\X$, that is, a smooth section of $T \X$. 
A distribution $D$ on $\X$ (a subbundle of $T\X$) is called invariant \cite{nvds} with respect to $f$ if $L_f D \subset D$, that is $L_fX \in D$ for any vector field $X$ in $D$. Here $L_f$ denotes Lie derivative with respect to $f$; i.e., $L_fX = [f, X]$.

In particular, the one-dimensional distribution spanned by a vector field $X$ on $\X$ is called invariant with respect to $f$ if there exists a function $\gamma$ on $\X$ such that
\begin{equation}
L_f X = \gamma X
\end{equation}
In local coordinates $x=(x^1, \cdots, x^n)$ for $\X$ and writing $f$ and $X$ as column vectors this amounts to the equality
\begin{equation}\label{1}
\frac{\partial f}{\partial x}(x)X(x) - \frac{\partial  X}{\partial x}(x)f(x) = - \gamma (x) X(x)
\end{equation}
(In \cite{kawanooud} this is expressed algebraically by saying that $X$ is a right eigenvector for $\frac{\partial f}{\partial x}(x)$ with eigenvalue $\gamma$.)
Similarly \cite{nvds}, a co-distribution $P$ on $\X$ (a subbundle of $T^*\X$) is called invariant if $L_fP \subset P$, that is, $L_f\alpha \in P$ for any one-form $\alpha$ on $\X$ (a one-form on $\X$ is a smooth section of $T^* \X$). In particular, the one-dimensional {\it co-distribution} spanned by a one-form $\alpha$ is said to be invariant with respect to $f$ if there exists a function $\gamma$ on $\X$ such that
\begin{equation}\label{2}
L_f \alpha = \gamma \alpha
\end{equation}
In local coordinates $x$, and expressing $\alpha$ as column vector, this amounts to the equality\footnote{Use the magical formula $L_f\alpha = L_f d \alpha + d( \alpha (f))$.}
\begin{equation}
\alpha^T (x)\frac{\partial f}{\partial x}(x) + (\frac{\partial  \alpha}{\partial x}(x)f(x))^T =  \gamma (x) \alpha^T(x)
\end{equation}
(In \cite{kawanooud} this is expressed by saying that $\alpha$ is a left eigenvector for $\frac{\partial f}{\partial x}(x)$ with eigenvalue $\gamma$.)

\subsection{Lifts of functions and vector fields to the tangent and cotangent bundle}
In this subsection it is recalled from \cite{Yals} (see also \cite{crouch},\cite{cortes}), how functions and vector fields on the state space manifold $\X$ can be lifted to functions, respectively vector fields, on its tangent and cotangent bundle. 

First we introduce the notions of {\it complete} and {\it vertical}
lifts of functions and vector fields to the {\it tangent bundle}. 

Given a function $h$ on $\X$, the \emph{complete
  lift} of $h$ to $T\X$, $\supscr{h}{c}:T\X \rightarrow \real$, is
defined by $\supscr{h}{c} (x,\dx) = \langle dh,\dx \rangle(x)$, with $\langle \cdot,\cdot \rangle$ denoting the duality pairing between elements of the co-tangent space and the tangent space at $x \in \X$. In local coordinates $x=(x^1,\ldots,x^n)$ for $\X$ and the induced local coordinates $(x,\dx) = (x^1,\ldots,x^n, \dx^1,\ldots, \dx^n)$ for $T\X$ this
reads
\begin{equation}
\supscr{h}{c} (x,\dx) = \sum_{a=1}^n \pder{h}{x^a}(x) \, \dx_a 
\end{equation}
The \emph{vertical lift} of a function $h$ to a function on $T\X$, $\supscr{h}{v} : T\X
\rightarrow \real$, is defined by $\supscr{h}{v}(x,\dx) = h \circ \tau_{\X}$,
where $\tau_{\X} : T\X \to \X$ denotes the tangent bundle projection $\tau_{\X}(x,\dx)=x$. In local induced coordinates $\supscr{h}{v}(x,\dx)=h(x)$.

Given a vector
field $f$ on $\X$, the \emph{complete lift} $\supscr{f}{c}$ of $f$ to $T\X$ is defined as the unique vector field
satisying $L_{\supscr{f}{c}}\supscr{h}{c} = \supscr{(L_f h)}{c}$, for any
function $h$ on $\X$ (with $L_fh$ denoting the Lie-derivative of the function $h$ along the vector field $f$, and similarly for $L_{\supscr{f}{c}}\supscr{h}{c}$).  Alternatively, if $\Phi_t : \X \rightarrow \X$,
$t \in [0,\eps)$, denotes the flow of $f$, then $\supscr{f}{c}$ is the vector field whose flow is given by
$(\Phi_t)_*: T\X \rightarrow T\X$. In induced local coordinates $(x^1,\ldots,x^n, \dx^1,\ldots, \dx^n)$ for $T\X$,
\begin{gather}\label{eq:complete-lift}
  \supscr{f}{c} (x,\dx) = \sum_{a=1}^n f_a(x) \pder{}{x^a} + \sum_{a,b=1}^n
  \pder{f_a}{x^b} (x) \dx^b \pder{}{(\dx^a)} 
\end{gather}
Finally, the \emph{vertical lift} $\supscr{f}{v}$ of $f$ to $T\X$ is the unique vector field on $T\X$ such that $L_{\supscr{f}{v}}
\supscr{h}{c} = \supscr{(L_fh)}{v}$, for any function $h$.
In induced local coordinates for $T\X$
\begin{gather}\label{eq:vertical-lift}
  \supscr{f}{v} (x,\dx) = \sum_{a=1}^n f_a(x) \pder{}{(\dx^a)} 
\end{gather}

Furthermore, the {\it vertical} and {\it Hamiltonian} lifts to the {\it co-tangent bundle} are defined as follows; see again \cite{Yals}. As before for the tangent bundle case, the \emph{vertical lift} $\supscr{h}{v} : T^*\X
\rightarrow \real$ of a function $h : \X \to \real$, is defined by $\supscr{h}{v} = h \circ \pi_{\X}$,
where $\pi_{\X} : T^*\X \to \X$ denotes the co-tangent bundle projection $\pi_{\X}(x,p)=x$. In induced local  coordinates $(x,p) := (x^1,\ldots,x^n, p_1,\ldots, p_n)$ for $T^*\X$ we have $\supscr{h}{v}(x,p)=h(x)$. 

Since there is a natural symplectic form on the cotangent bundle $T^* \X$ we can define the Hamiltonian vector field on $T^*\X$ corresponding to $\supscr{h}{v}$, denoted by $X_{\supscr{h}{v}}$, and called the {\it vertical Hamiltonian lift}. In induced local coordinates $(x,p)$ for $T^*\X$
\begin{equation}\label{Hamvertical-lift}
X_{\supscr{h}{v}} = - \sum_{a=1}^n \frac{\partial h}{\partial x_a}(x) \pder{}{(p^a)} 
\end{equation}
Furthermore, for any vector field $f$ on $\X$ define the {\it Hamiltonian function} $H^f : T^* \X \to \mathbb{R}$ as
\begin{equation}
H^f(x,p) = \langle p, f(x) \rangle = p^Tf(x)
\end{equation}
The corresponding Hamiltonian vector field on $T^*\X$, denoted by $X_{H^f}$, is called the {\it complete Hamiltonian lift}. 
In induced local coordinates for $T^*\X$
\begin{equation}
X_{H^f} = \sum_{a=1}^n f_a(x) \pder{}{x^a} - \sum_{a,b=1}^n
  \pder{f_b}{x^a} (x) p_b \pder{}{(p_a)} 
\end{equation}
For later use we mention that given a lift (complete or vertical) of a vector field $f$ to the tangent bundle, as well as a lift (Hamiltonian or vertical) to the co-tangent bundle, we can {\it combine} the two lifts into a vector field defined on the Whitney sum $T \X \oplus T^* \X$ (that is, the base manifold $\X$ together with the fiber space $T_x \X \times T_x^* \X$ at any point $x \in \X$). An example (to be used later on) is the combination of the complete lift $\supscr{f}{c}$ on $T\X$ with the Hamiltonian extension $X_{H^f}$ on $T^*\X$, which defines a vector field on $T \X \oplus T^* \X$, which will be denoted as $\supscr{f}{c} \oplus X_{H^f}$. Furthermore, since the vertical lifts $\supscr{f}{v}$ (to $T\X$) and $X_{\supscr{H}{v}}$ (to $T^*\X$) do not have components on the base manifold $\X$ we may also define the combined vector field $\supscr{f}{v} \oplus X_{\supscr{h}{v}}$ on the Whitney sum $T \X \oplus T^* \X$ for any vector field $f$ and function $h$.

\subsection{Prolongation of nonlinear control systems to the tangent and the co-tangent bundle}
Armed with the notions of the lifts of functions and vector fields to tangent and cotangent bundle as described in the previous subsection, we now recall from \cite{crouch}, see also \cite{cortes, vdsdiffpass}, how we can define prolongations of {\it nonlinear control systems} to tangent and cotangent bundles.

Consider a nonlinear control system $\Sigma$ with state space $\X$, affine in the
inputs $u$, and with an equal number of outputs $y$, given as
\begin{equation}\label{eq:system}
  \Sigma : \left\{
    \begin{array}{l}
      \displaystyle{\dot{x} = f(x) + \sum_{j=1}^m u_j g_j (x)} \\
      y_j = h_j(x) \, , \quad j=1,\ldots,m \, , 
    \end{array}
  \right.
\end{equation}
where $x \in \X$, and $u = (u_1,\ldots,u_m) \in \U \subset
\real^m$. 
%The vector fields $f, g_1,\ldots,g_m$ on $\X$ are assumed to
%be complete, and $h_1,\ldots,h_m$ are real-valued functions on $\X$. 
The
set $\U$ is the input space, which is assumed to be an
open subset of $\real^m$. 
%The function $t \mapsto u(t)
%= ( u_1(t), \ldots, u_m(t))$, that we will commonly denote as
%$u(\cdot)$, belongs to a certain class of functions of time, called the set of \emph{admissible controls}. 
Finally, $\Y = \mathbb{R}^m$ is the output space.

The prolongation of the nonlinear control system to the tangent bundle and the cotangent bundle is constructed as follows; cf. \cite{crouch}.

Given an initial state $x(0)=x_0$, take any coordinate neighborhood of
$\X$ containing $x_0$. Let $t \in [0,T] \mapsto x(t)$ be the solution
of~\eqref{eq:system} corresponding to the admissible input function $t \in [0,T]
\mapsto u(t) = (u_1(t),\ldots,u_m(t))$ and the initial state
$x(0)=x_0$, such that $x(t)$ remains within the selected coordinate
neighborhood.  Denote the resulting output by $t \in [0,T] \mapsto
y(t)=(y_1(t),\ldots,y_m(t))$, with $y_j(t)=H_j(x(t))$.  Then the
\emph{variational system} along the input-state-output trajectory $t
\in [0,T] \mapsto (x(t),u(t),y(t))$ is given by the following
time-varying system
\begin{equation}\label{eq:variational}
\begin{array}{rcl}
  \dot{\dx}(t) &=& \pder{f}{x} (x(t)) \dx(t) + \\[2mm]
   && \sum_{j=1}^m u_j(t)
  \pder{g_j}{x}(x(t)) \dx(t) + \\[2mm]
  && \sum_{j=1}^m \du_j(t) g_j(x(t)) \\[2mm]
  \dy_j (t) & = & \pder{h_j}{x} (x(t)) \dx(t) \, , \quad j=1,\ldots,m \, ,
  \end{array}
\end{equation}
with state $\dx(t) \in T_{x(t)}^*\X$, where $\du = (\du_1, \ldots, \du_m)^T$,
$\dy = (\dy_1,\ldots, \dy_m)^T$ denote the input and output vectors of the
variational system. (Note that $\pder{h_j}{x} (x)$ denotes a row vector.) 

The reason behind the terminology `variational'
comes from the following fact: let $(x(t,\eps),u(t,\eps),y(t,\eps))$,
$t\in [0,T]$, be a family of input-state-output trajectories
of~\eqref{eq:system} parameterized by $\eps \in (-\delta, \delta)$,
with $x(t,0)=x(t)$, $u(t,0)=u(t)$ and $y(t,0)=y(t)$, $t \in [0,T]$.
Then, the infinitesimal variations
\[
\dx(t) = \pder{x}{\eps}(t,0) \, , \quad \du(t) = \pder{u}{\eps}(t,0) \,
, \quad \dy(t) = \pder{y}{\eps}(t,0) \, ,
\]
satisfy equation~\eqref{eq:variational}.
\begin{remark} 
For a {\it linear} system $\dot{x} = Ax + Bu, y=Cx$ the variational systems along any trajectory are simply given as 
$\dot{\dx} = A\dx + B\du, \dy=C\dx$.
\end{remark}

The \emph{prolongation} (or {\it prolonged system}) of~\eqref{eq:system}
comprises the original
system~\eqref{eq:system} {\it together} with its variational systems, that is the total system
\begin{equation}\label{eq:prolonged-system-coordinates}
\begin{array}{rcl}
\dot{x} & = & f(x) + \sum_{j=1}^m u_j g_j (x)  \\[2mm]
\dot{\dx}(t) & = & \pder{f}{x} (x(t)) \dx(t) + \\[2mm]
&& \sum_{j=1}^m u_j(t)
  \pder{g_j}{x}(x(t)) \dx(t) + \\[2mm]
 && \sum_{j=1}^m \du_j(t) g_j(x(t)) 
  \\[3mm]
y_j & = & h_j(x), \quad j=1,\ldots,m  \\[2mm]
\dy_j (t) & = & \pder{h_j}{x} (x(t)) \, \dx(t), \quad j=1,\ldots, \, , 
\end{array}
\end{equation}
with inputs $u_j$, $\du_j$, outputs $y_j$, $\dy_j, j=1, \cdots, m,$ and state vector
$x$, $\dx$. 

Using the previous subsection the prolonged
system~\eqref{eq:prolonged-system-coordinates} on the tangent
space $T\X$ can be defined in the following {\it coordinate-free} way. Denote the elements of $T\X$ by $x_l=(x,\dx)$, where $\tau_{\X}(x_l)= x \in \X$ with $\tau_{\X}: T \X \to \X$ again the tangent bundle projection.

\begin{definition}\cite{crouch}\label{dfn:prolonged-system}
  The prolonged system $\dSigma$ of a nonlinear system $\Sigma$ of the
  form~\eqref{eq:system} is defined as the system
  \begin{equation}\label{eq:prolonged-system}
    \dSigma : \left\{
      \begin{array}{l}
        \displaystyle{\dot{x}_l = \supscr{f}{c} (x_l) + \sum_{j=1}^m u_j\supscr{g_j}{c} (x_l) +
         \sum_{j=1}^m \du_j \supscr{g_j}{v} (x_l)} \\ 
        y_j = \supscr{h_j}{v} (x_l) \, , \quad j=1,\ldots,m \\[2mm]
        \dy_j =\supscr{h_j}{c} (x_l) \, , \quad j=1,\ldots,m 
      \end{array}
    \right.
  \end{equation}
  with state $x_l=(x,\dx) \in T \X$, inputs $u_j, \du_j$ and outputs $y_j, \dy_j$, $j=1, \ldots, m$.
\end{definition}
Note that the prolonged system $\dSigma$ has state space $T\X$, input space $T\U$ and output space $T\Y$. One can easily check that in any system of local coordinates $x$ for $\X$ and the induced local coordinates $x,\dx$ for $T\X$,
the local expression of the system~\eqref{eq:prolonged-system} equals~\eqref{eq:prolonged-system-coordinates}.
\begin{remark}
For a linear system $\dot{x} = Ax + Bu, y=Cx$ the prolonged system is simply the {\it product} of the system with the copy system $\dot{\dx} = A\dx + B\du, \dy=C\dx$.
\end{remark}

The prolongation of the nonlinear control system $\Sigma$ to the {\it co-tangent bundle} is defined as follows. Associated to the variational system (\ref{eq:variational}) there is the {\it adjoint variational system}, defined as
\begin{equation}\label{eq:adjointvariational}
\begin{array}{rcl}
  \dot{p}(t) &=& -(\frac{\partial f}{\partial x})^T (x(t)) p(t)  \\[2mm]
   && - \sum_{j=1}^m u_j(t)
  (\frac{\partial g_j}{\partial x})^T(x(t)) p(t) \\[2mm]
 && - \sum_{j=1}^m du_j(t) \frac{\partial^T h_j}{\partial x}(x(t)) \\[2mm]
  dy_j (t) & = & p^Tg_j(x(t))  \, , \quad j=1,\ldots,m \, ,
  \end{array}
\end{equation}
with state variables $p \in T^*_{x(t)}\X$, and adjoint variational inputs and outputs $du_j, dy_j, j=1,\ldots,m$.

Then the original nonlinear system $\Sigma$ together with all it adjoint variational systems defines the total system
\begin{equation}\label{eq:prolonged-system-coordinatesH}
\begin{array}{rcl}
\dot{x} & = & f(x) + \sum_{j=1}^m u_j g_j (x)  \\[2mm]
\dot{p}(t) & = & - (\frac{\partial f}{\partial x})^T(x(t))p(t) \\[2mm]
&& - \sum_{j=1}^m u_j(t) (\frac{\partial g_j}{\partial x})^T(x(t)) p(t)  \\[2mm]
&& - \sum_{j=1}^m du_j(t) \frac{\partial^Th_j}{\partial x}(x(t))
  \\[3mm]
y_j & = & h_j(x), \quad j=1,\ldots,m  \\[2mm]
 dy_j (t) & = & p^Tg_j(x(t))  \, , \quad j=1,\ldots,m \, ,
\end{array}
\end{equation}
with inputs $u_j$, $du_j$, outputs $y_j$, $dy_j, j=1, \cdots, m,$ and state
$x$, $p$. This total system is called the {\it Hamiltonian extension}. In a coordinate-free way the Hamiltonian extension is defined as follows.
\begin{definition}\cite{crouch}\label{dfn:Hamiltonianextension}
  The Hamiltonian extension $d\Sigma$ of a nonlinear system $\Sigma$ of the
  form~\eqref{eq:system} is defined as the system
  \begin{equation}\label{Hamextension}
    d\Sigma : \left\{
      \begin{array}{rcl}
        \dot{x}_e & = & X_{H^f} (x_e) + \sum_{j=1}^m u_j X_{H^{g_j}} (x_e) + \\[2mm]
         && \sum_{j=1}^m du_j X_{\supscr{h}{v}_j} (x_e)  \\[2mm] 
        y_j &=& \supscr{h}{v}_j (x_e) \, , \quad j=1,\ldots,m \\[2mm]
        dy_j &=&H^{g_j} (x_e) \, , \quad j=1,\ldots,m 
      \end{array}
    \right.
  \end{equation}
  with state $x_e =(x,p) \in T^* \X$, inputs $u_j, du_j$ and outputs $y_j, dy_j$, $j=1, \ldots, m$.
\end{definition}
Note that the Hamiltonian extension $d\Sigma$ has state space $T^*\X$, input space $T^*\U$ and output space $T^*\Y$. 
\begin{remark}
For a linear system $\dot{x} = Ax + Bu, y=Cx$ the Hamiltonian extension is simply the product of the system with its adjoint system $\dot{p} = -A^Tp - C^T du, dy=B^Tp$.
\end{remark}
\begin{remark}
The prolongation $\dSigma$ of $\Sigma$ to the tangent bundle can be {\it combined} with the Hamiltonian extension $d\Sigma$ of $\Sigma$ to the co-tangent bundle. This will define a system on the Whitney sum $T \X \oplus T^* \X$ with inputs $u, \du, du$, states $x, \dx, p$ and outputs $y,\dy, dy$.
\end{remark}

\section{Invariant subbundles}
Let as before $\X$ denote the $n$-dimensional state space manifold. 
Consider the Whitney sum $T \X \oplus T^* \X$ (as explained before, the base manifold $\X$ together with the fiber space $T_x \X \times T_x^* \X$ at any point $x \in \X$).
\begin{definition}
A {\it subbundle} $K$ of $T \X \oplus T^* \X$ is a vector bundle over $\X$ with fiber $K(x) \subset T_x \X \times T_x^* \X$ at any point $x \in \X$. The subbundle $K$ is called {\it invariant} with respect to a vector field $f$ on $\X$ if
\begin{equation}\label{invariance}
(L_f X, L_f \alpha) \in K \, \mbox{ for any } (X, \alpha) \in K
\end{equation}
\end{definition}

\smallskip

\begin{remark}
If $K$ has only zero components in $T_x^* \X$ for any point $x \in \X$, then $K$ can be identified with a {\it distribution} on $\X$. Alternatively, if $K$ has only zero components in $T_x \X$ for any point $x \in \X$, then it can be regarded as a {\it co-distribution} on $\X$. In these cases invariance of $K$ with respect to $f$ amounts to invariance of the associated distribution, respectively, co-distribution, with respect to $f$.
\end{remark}
\begin{remark}
The above definition of invariance of $K$ is formally identical to the definition of an {\it infinitesimal symmetry} of a {\it Dirac structure}; see \cite{Courant, dorfman, vdssym} for details. (A Dirac structure is a subbundle of $T \X \oplus T^* \X$  which is maximally isotropic with respect to the duality product.)
\end{remark}

Associated to the subbundle $K$ we can define the {\it submanifold} $\K$ of $T \X \oplus T^* \X$ as follows
\begin{equation}
\K := \{ (x, \dx, p) \in T \X \oplus T^* \X \mid (\dx,p) \in K(x) \}
\end{equation}
We have the following useful characterization of invariance of invariance of $K$. First, recall from the previous section that we may lift the vector field $f$ to the vector field $\supscr{f}{c}$ on $T\X$ (the complete lift), as well as to the vector field $X_{H^f}$ on $T^*\X$ (the complete Hamiltonian lift). Taken together this results in a lift to a vector field on the Whitney sum $T \X \oplus T^* \X$, denoted as
\begin{equation}\label{fextended}
\supscr{f}{c} \oplus X_{H^f}
\end{equation}
\begin{proposition}
The subbundle $K$ is invariant with respect to the vector field $f$ on $\X$ if and only if the submanifold $\K$ is invariant for the vector field $\supscr{f}{c} \oplus X_{H^f}$.
\end{proposition}
\begin{proof}
In coordinates $x$ and induced local coordinates for $T\X$ and $T^*\X$ the vector field $\supscr{f}{c} \oplus X_{H^f}$ at a point $(x,X(x),\alpha(x)) \in \K \subset T \X \oplus T^* \X $ is given by
\[
\begin{array}{rcl}
\begin{bmatrix}
f(x) \\ \frac{\partial f}{\partial x}(x) X(x) \\ - (\frac{\partial f}{\partial x})^T(x) \alpha(x)	 
\end{bmatrix}
& = &
\begin{bmatrix}
f(x) \\ \frac{\partial X}{\partial x}(x) f(x) - L_f X(x) \\  \frac{\partial \alpha}{\partial x}(x) f(x)) - L_f \alpha(x) 
\end{bmatrix} \\[6mm]
& = &
\begin{bmatrix}
f(x) \\ \frac{\partial X}{\partial x}(x) f(x)  \\  \frac{\partial \alpha}{\partial x}(x) f(x)
\end{bmatrix}
-
\begin{bmatrix}
0 \\ L_f X(x) \\ L_f \alpha(x)
\end{bmatrix} \, ,
\end{array}
\]
where the first vector in the last term denotes a tangent vector to $\K$.
Thus if (\ref{invariance}) holds then the vector field $\supscr{f}{c} \oplus X_{H^f}$ is tangent to $\K$. Conversely, if $\supscr{f}{c} \oplus X_{H^f}$ is tangent to $\K$ then this implies that the second vector in the last term is tangent to $\K$ for all $(X(x),\alpha(x)) \in K(x)$, which amounts to (\ref{invariance}), i.e., invariance of $K$. 
\end{proof}

\section{Differential Hamiltonian systems}
It is well-known, see, e.g., \cite{vanderschaftbook}, that the Hamiltonian equations arising from applying Pontryagin's Maximum principle to the optimal control problem of minimizing the cost criterion
\begin{equation}\label{opt}
 \frac{1}{2} \int_0^{\infty} \left( \|y(t)\| ^2 + \|u(t)\| ^2\right) dt
 \end{equation}
 for the nonlinear control system $\Sigma$ are given by the following system on $T^* \X$
 \begin{equation}\label{Hamopt}
 \begin{array}{rcl}
 \dot{x} & = & \frac{\partial^T H^{\mathrm{opt}}}{\partial p}(x,p) \\[2mm]
 \dot{p} & = & - \frac{\partial^T H^{\mathrm{opt}}}{\partial x}(x,p)
 \end{array}
 \end{equation}
 with the Hamiltonian $H^{\mathrm{opt}}: T^* \X \to \mathbb{R}$ given by
 \begin{equation}
 H^{\mathrm{opt}}(x,p) = p^Tf(x) - \frac{1}{2}p^Tg(x)g^T(x)p + \frac{1}{2} h^T(x)h(x),
 \end{equation}
where $g$ has columns $g_1, \ldots g_m$ and $h: \X \to \mathbb{R}^m$ has components $h_1, \ldots, h_m$.
 
Motivated by \cite{kawano} we will study the differential version of the Hamiltonian system (\ref{Hamopt}). In the next section we will apply this to the differential version of the Hamilton-Jacobi equation corresponding to (\ref{Hamopt}), called a differential Riccati equation in \cite{kawano}.

In order to motivate the subsequent developments, let us first recall from e.g. \cite{vanderschaftbook}, see also \cite{crouch}, that in the {\it linear} case the Hamiltonian system (\ref{Hamopt}) on $T^*\X$ can be obtained from interconnecting the linear system $\dot{x} = Ax + Bu, y=Cx$ with the adjoint system $\dot{p} = -A^Tp - C^T du, dy = B^Tp$, by the interconnection equations
\[
u= - dy, \, du= y, 
\]
leading to the {\it linear} Hamiltonian system
\begin{equation}
\begin{bmatrix} \dot{x} \\ \dot{p} \end{bmatrix} =
\begin{bmatrix} A & -BB^T \\ -C^TC & -A^T \end{bmatrix} 
\begin{bmatrix} x \\ p \end{bmatrix}
\end{equation}
Similarly, we will now define a {\it differential Hamiltonian system} by considering the interconnection of the prolongation $\dSigma$ given in (\ref{eq:prolonged-system-coordinates}) and the Hamiltonian extension $d\Sigma$ in (\ref{eq:prolonged-system-coordinatesH}), via the interconnection equations on the (adjoint) variational inputs and outputs
\[
\du = - dy, \, du = \dy
\]
It can be directly verified from the definition of the prolongation $\dSigma$ and the Hamiltonian extension $d\Sigma$ that this defines a system on the Whitney sum $T \X \oplus T^* \X$ which in a coordinate-free fashion is given as
\begin{equation}\label{diffHam}
\begin{array}{rcl}
\dot{z} & = &  \supscr{f}{c} \oplus X_{H^f}(z) - \sum_{j=1}^m H^{g_j} \supscr{g}{v}_j \oplus \supscr{h}{c}_j X_{\supscr{h_j}{v}}(z)  \\[2mm]
&& + \sum_{j=1}^m u_j  \supscr{g}{c}_j \oplus X_{H^{g_j}}(z) \\[2mm]
      y_j &= &h_j(x) \, , \quad j=1,\ldots,m \, ,
      \end{array}
\end{equation}
with total state $z:=(x,\dx,p)$, and with remaining inputs and outputs $u_j,y_j, j=1, \ldots,m$. Note that $\supscr{g}{v}_j \oplus X_{\supscr{h_j}{v}}$ is the vector field on the Whitney sum $T \X \oplus T^* \X$ obtained from combining the vector field $\supscr{g}{v}_j$ on $T \X$ with the vector field $X_{\supscr{h_j}{v}}$ on $T^*\X$. The system (\ref{diffHam}) is called the {\it differential Hamiltonian system}\footnote{Note however that the system (\ref{diffHam}) by itself is {\it not} Hamiltonian in an ordinary sense. However it can be interpreted as a linear Hamiltonian system along trajectories of $\Sigma$.}.

Similarly to what we did before for invariance with respect to $\supscr{f}{c} \oplus X_{H^f}$ we can define invariance of subbundles with respect to the differential Hamiltonian system (\ref{diffHam}).
\begin{definition}
Consider the differential Hamiltonian system (\ref{diffHam}) on the Whitney sum $T \X \oplus T^* \X$. A subbundle $K$ of $T \X \oplus T^* \X$ with its associated submanifold $\mathcal{K} \subset T \X \oplus T^* \X$ is invariant with respect to (\ref{diffHam}) if 
\begin{equation}\label{diffHam1}
\supscr{f}{c} \oplus X_{H^f} - \sum_{j=1}^m H^{g_j} \supscr{g}{v}_j \oplus \supscr{h}{c}_j X_{\supscr{h}{v}_j}  + \sum_{j=1}^m u_j  \supscr{g}{c}_j \oplus X_{H^{g_j}}
\end{equation}
is tangent to $\mathcal{K}$ for all $u_j,j=1, \ldots,m$.
\end{definition}
The following proposition is immediate.
\begin{proposition}\label{propinv}
A subbundle $K$ of $T \X \oplus T^* \X$ with its associated submanifold $\mathcal{K} \subset T \X \oplus T^* \X$ is called invariant with respect to (\ref{diffHam}) if and only if \\
$(1)$ $\supscr{f}{c} \oplus X_{H^f} - \sum_{j=1}^m H^{g_j} \supscr{g}{v}_j \oplus \supscr{h}{c}_j X_{\supscr{h_j}{v}}$ is tangent to $\mathcal{K}$,\\
$(2)$ $K$ is invariant for $g_j, j=1, \ldots,m$.
\end{proposition}

\section{Invariant Lagrangian subbundles and differential Riccati equations}
Associated to the optimal control problem (\ref{opt}) and the resulting Hamiltonian system (\ref{Hamopt}) there is the Hamilton-Jacobi-Bellman equation 
\begin{equation}\label{hamjac}
\begin{array}{c}
 \frac{\partial P}{\partial x}(x)f(x) - \frac{1}{2}\frac{\partial P}{\partial x}(x)g(x)g^T(x)\frac{\partial^T P}{\partial x}(x) \\[2mm] + \frac{1}{2} h^T(x)h(x) =0
 \end{array}
 \end{equation}
Under appropriate conditions, the positive solution to the Hamilton-Jacobi equation is the {\it value function} $P: \X \to \mathbb{R}$ of the optimal control problem, that is, $P(x)$ is the minimal cost for the system starting at time $0$ at initial state $x$. Furthermore, the optimal control is given in feedback form as $u= -g^T(x) \frac{\partial^T P}{\partial x}(x)$, while the {\it Lagrangian} submanifold $N := \{(x,p) \in T^* \X \mid p= \frac{\partial^T P}{\partial x}(x) \}$, see e.g. \cite{abraham},  equals the stable invariant manifold of the Hamiltonian system (\ref{Hamopt}), cf. \cite{vanderschaftbook}. In the linear case $\dot{x} = Ax + Bu, y=Cx$ the Hamilton-Jacobi equation (\ref{hamjac}) reduces to the well-known Riccati equation
\begin{equation}\label{linRic}
A^TP + PA - PBB^TP + C^TC=0,
\end{equation}
whose positive solution $P$ yields the quadratic value function $P(x) = \frac{1}{2} x^TPx$, and is such that the Lagrangian subspace $\{(x,p) \mid p = Px \}$ is the generalized eigenspace corresponding to the $n$ eigenvalues in the left-half of the complex plane (assuming, e.g., minimality of $(A,B,C)$).

Recently in \cite{kawano}, motivated in particular by developments in \cite{manchester, manchester1}, the {\it differential version} of the Hamilton-Jacobi equation (\ref{hamjac}) was introduced, called a (generalized) differential Riccati equation. In this section we will approach this from a coordinate-free point of view, using the machinery built up in the previous sections.

In order to do so we will define a special type of subbundle of the Whitney sum $T \X \oplus T^* \X$.
\begin{definition}
A subbundle $K$ of $T \X \oplus T^* \X$ is called a {\it Lagrangian subbundle} if $K(x) \subset T_x \X \times T_x^* \X$ is a Lagrangian subspace (with respect to the canonical symplectic form on $T_x \X \times T_x^* \X$ \cite{abraham}) for every $x \in \X$.
\end{definition}
\begin{example} All subbundles $K$ with $K(x) = \{ (\dx,p) \mid p = \Pi(x)\dx \}$, where $\Pi(x)$ is a {\it symmetric} matrix, are Lagrangian. More generally, all subbundles
\begin{equation}
\begin{array}{c}
K(x)  = \{ (\dx,p) \mid V(x)p = U(x)\dx, \\[3mm] 
V(x)U^T(x)  = U(x)V^T(x),  \rank \begin{bmatrix} U(x) & V(x) \end{bmatrix} = n\}
\end{array}
\end{equation}
with $U(x), V(x)$ $n \times n$ matrices depending on $x$, are Lagrangian.
\end{example}

An important special class of Lagrangian subbundles is defined as follows. Let $N \subset T^* \X$ be a {\it Lagrangian submanifold}, given as $N = \{(x,p) \in T^* \X \mid p  = \frac{\partial^T P}{\partial x}(x) \}$ for some (generating) function $P: \X \to \mathbb{R}$. Then tangent vectors to $N$ at a point $(x,p) \in N$ are vectors $ (X(x), \alpha(x)) \in T_{(x,p)} N$ with $\alpha(x) = \frac{\partial^2 P}{\partial x^2}(x)X(x)$,
where $X(x) \in T_x\X$ and $\alpha(x) \in T_pT_x^*\X$, and with $\frac{\partial^2 P}{\partial x^2}(x)$ denoting the Hessian matrix of $P$. Identifying $T_pT_x^*\X$ with $T_x^*\X$ (well-defined since $T_x^*\X$ is a linear space), this yields the Lagrangian subbundle
\begin{equation}\label{intLag}
K(x) = \{ (\dx,p) \mid p = \frac{\partial^2 P}{\partial x^2}(x) \dx \}
\end{equation}
Such Lagrangian subbundles will be called {\it integrable} Lagrangian subbundles. Local integrability of Lagrangian subbundles can be characterized as follows. A Lagrangian subbundle $K(x) = \{ (\dx,p) \mid p = \Pi(x)\dx \}$ is integrable if and only if there exists a function $P: \X \to \mathbb{R}$ such that
\begin{equation}\label{integrabilityR1}
\pi_{ij}(x) = \frac{\partial^2 P}{\partial x_i \partial x_j}(x), \quad i,j=1, \cdots,n
\end{equation}
A necessary and sufficient condition for the local existence of such a function $P$ is the integrability condition (see \cite{duistermaat} for the same condition in the characterization of Hessian Riemannian metrics)
\begin{equation}\label{integrabilityR}
\frac{\partial \pi_{jk}}{\partial x_i}(x) = \frac{\partial \pi_{ik}}{\partial x_j}(x), \quad i,j,k=1, \cdots,n
\end{equation}
Indeed, (\ref{integrabilityR}) guarantees the local existence of functions $p_k(x)$ such that $\pi_{jk}(x) = \frac{\partial p_k}{\partial x_j}(x), j,k=1, \cdots,n$. Then by symmetry of $\Pi$
\begin{equation}
\frac{\partial p_k}{\partial x_j}(x) = \pi_{jk}(x) = \pi_{kj}(z) = \frac{\partial p_j}{\partial x_k}(x), \quad j,k=1, \cdots,n
\end{equation}
which is the integrability condition guaranteeing the local existence of a function $P(x)$ satisfying
\begin{equation}\label{int}
p_j(x) = \frac{\partial P}{\partial x_j}(x), \quad j=1,\cdots,n
\end{equation}
By differentiation of (\ref{int}) with respect to $x_i$ and in view of the definition of $p_j(x), j=1,\cdots, n,$ this amounts to (\ref{integrabilityR1}). 

\smallskip

Now consider a Lagrangian subbundle $K$ of $T \X \oplus T^* \X$ which is invariant for the system (\ref{diffHam}), i.e., by Proposition \ref{propinv} $\supscr{f}{c} \oplus X_{H^f} - \sum_{j=1}^m H^{g_j} \supscr{g}{v}_j \oplus \supscr{h}{c}_j X_{\supscr{h_j}{v}}$ is tangent to $\mathcal{K}$ and $K$ is invariant for $g_j, j=1, \ldots,m$. 

Additionally assume that the projection of $K(x) \subset T_x \X \oplus T_x^* \X$ on $T_x \X$ is equal to the whole tangent space $T_x \X$ for all $x \in \X$. Then in any set of local coordinates $x^1, \cdots, x^n$ for $\X$ the Lagrangian subbundle $K$ is spanned by pairs of vector fields and one-forms
\[
(\frac{\partial}{\partial x_i}, \pi_i), \quad i=1, \cdots, n
\]
where the one-forms
\[
\pi_i(x) = \pi_{1i}(x) dx^I + \cdots \pi_{ni}(x) dx^n, \quad  i=1, \cdots, n
\]
satisfy, because of the fact that $K$ is Lagrangian, the symmetry property
\[
\pi_{ji}(x) = \pi_{ij}(x) , \quad i,j=1, \cdots, n
\]
Defining the $n \times n$ symmetric matrix $\Pi(x)$ with $(i,j)$-th element $\pi_{ij}$ it immediately follows, cf. (\ref{1}) and (\ref{2}), that invariance of $K$ with respect to the system (\ref{diffHam}) amounts to the coordinate expression ($\supscr{f}{c} \oplus X_{H^f} - \sum_{j=1}^m H^{g_j} \supscr{g}{v}_j \oplus \supscr{h}{c}_j X_{\supscr{h_j}{v}}$ is tangent to $\mathcal{K}$)
\begin{equation}\label{3}
\begin{array}{l}
(\frac{\partial f}{\partial x})^T(x)\Pi(x) + \Pi(x)\frac{\partial f}{\partial x}(x) 
- \Pi(x)g(x)g^T(x)\Pi(x) \\[3mm]
+ (\frac{\partial h}{\partial x})^T(x)\frac{\partial h}{\partial x}(x) + \frac{\partial \Pi}{\partial x}(x)f(x)=0,
\end{array}
\end{equation}
together with ($K$ is invariant for $g_j, j=1, \ldots,m$)
\begin{equation}\label{4}
\begin{array}{r}
(\frac{\partial g_j}{\partial x})^T(x)\Pi(x) + \Pi(x)\frac{\partial g_j}{\partial x}(x) + \frac{\partial \Pi}{\partial x}(x)g_j(x)= 0, \\[2mm]
 j=1, \cdots, m
 \end{array}
\end{equation}
The equation (\ref{3}) is called in \cite{kawano} the (generalized) {\it differential Riccati equation}.
\begin{remark}
Note that the equation (\ref{4}) is not present in \cite{kawano} since in that paper throughout the assumption is made (continuing upon similar assumptions in \cite{manchester, manchester1}) that the vector fields $g_j$ are independent of $x$ (that is, constant in the chosen local coordinates $x$) and furthermore that $\frac{\partial \Pi}{\partial x}(x)g_j= 0, \, j=1, \cdots, m$. In that case (\ref{4}) is trivially satisfied.
\end{remark}
\begin{remark}
In \cite{kawano} it is proved, analogousy to the linear case (\ref{linRic}), that an $n$-dimensional subbundle $K$ which is invariant with respect to $\supscr{f}{c} \oplus X_{H^f} - \sum_{j=1}^m H^{g_j} \supscr{g}{v}_j \oplus \supscr{h}{c}_j X_{\supscr{h}{v}_j}$ and has eigenvectors in the left-half of the complex plane is necessarily Lagrangian.
\end{remark}
\begin{remark}
Another connection is to the work on {\it state-dependent} Riccati equations; see, e.g., \cite{cimen} and the references quoted therein.
\end{remark}

\section{Differential Lyapunov equations}

Differential Lyapunov equations correspond to the case where $g_j=0, j=1, \cdots, m$, for the system $\Sigma$ (no inputs), and the interconnection of the resulting prolongation $\dSigma$ and Hamiltonian extension $d \Sigma$ reduces to $du = \dy$. (Notice that the variational inputs $\du$ and adjoint variational outputs $dy$ are absent.) This leads to the simplified differential Hamiltonian system (compare with (\ref{diffHam}))
\begin{equation}
\begin{array}{rcl}
\dot{z} & = &  \supscr{f}{c} \oplus X_{H^f}(z) - \sum_{j=1}^m \supscr{h}{c}_j X_{\supscr{h}{v}_j}(z)  \\[2mm]
      y_j &= & h_j(x) \, , \quad j=1,\ldots,m \, ,
      \end{array}
\end{equation}
while the differential Riccati equation (\ref{3}) simplifies to 
\begin{equation}
\begin{array}{l}
(\frac{\partial f}{\partial x})^T(x)\Pi(x) + \Pi(x)\frac{\partial f}{\partial x}(x) \\[3mm]
+ (\frac{\partial h}{\partial x})^T(x)\frac{\partial h}{\partial x}(x) + \frac{\partial \Pi}{\partial x}(x)f(x)=0 \, ,
\end{array}
\end{equation}
and (\ref{4}) is void. This is nothing else than the standard type of equation considered in contraction analysis \cite{lohmiller, jouffroy}.

On the other hand, an extension of the differential Hamiltonian system (\ref{diffHam}) concerns the differential version of the state feedback $H_{\infty}$ problem (see, e.g., \cite{manchester1, vanderschaftbook}), in which case there are, next to the input vector fields $g_j, j=1, \cdots, m$, additional {\it disturbance} vector fields. This relates to previous work on differential $L_2$-gain; see \cite{manchester1, vdsdiffpass}.

\section{Conclusions and outlook}
We have described a geometric framework for defining differential Hamiltonian systems and (generalized) differential Riccati equations. This already enabled the consideration of arbitrary input vector fields. The precise implications of this framework are yet to be seen, which is a topic of current research. In particular, the notions of 'integrability' of differential Hamiltonian systems and Riccati equations need further study, as well as the implications towards properties of incremental stabilizability, see, e.g.,  \cite{manchester,sepulchreforni1,kawano}.

\end{document}